\documentclass[11pt]{article}

\usepackage[margin=1in]{geometry} % to change the page dimensions
\usepackage{amsmath,amssymb,amsfonts,amsthm}
\usepackage{graphicx}
\usepackage{float}
\usepackage{setspace}
\usepackage{caption}
\usepackage{subcaption}

\usepackage{xcolor}
\usepackage[authoryear, round]{natbib}
\usepackage{hyperref}
\hypersetup{colorlinks,citecolor=blue,linkcolor=blue,breaklinks=true}
 \usepackage{epstopdf}
\usepackage{yhmath} % \wideparen 弧 \adots
\usepackage{makecell}
\usepackage{enumerate}
\usepackage{enumitem}
\usepackage{natbib}

\usepackage{aliascnt}
\allowdisplaybreaks

\newtheorem{theorem}{Theorem}[section]

\newtheorem{lemma}{Lemma}[section]

\newaliascnt{corollary}{theorem}

\aliascntresetthe{corollary}

\newaliascnt{proposition}{theorem}
\newtheorem{proposition}[proposition]{Proposition}
\aliascntresetthe{proposition}

\newtheorem{definition}{Definition}[section]

\newtheorem{remark}{Remark}

\newaliascnt{assumption}{theorem}

\aliascntresetthe{assumption}

\newcommand{\bbm}{\begin{bmatrix}}
	\newcommand{\ebm}{\end{bmatrix}}

\numberwithin{equation}{section}

\begin{document}
	
\title{\huge On the Square Root of Wishart Matrices: Exact Distributions and Asymptotic Gaussian Behavior }

 \author{
	Fengcheng Liu\footnote{Email: lfc2022@mail.ustc.edu.cn}
    \\
    \\
    Department of Probability and Statistics, School of Mathematical Sciences, \\
    University of Science and Technology of China
}

\date{}
	
\maketitle
	
\begin{abstract}
Random matrix theory has become a cornerstone in modern statistics and data science, providing fundamental tools for understanding high-dimensional covariance structures. Within this framework, the Wishart matrix plays a central role in multivariate analysis and related applications.
  This paper investigates both the exact and asymptotic distributions of the square root of a standard Wishart matrix. We first derive the exact distribution of the square root matrix. Then, by leveraging the Bartlett decomposition, we establish the joint asymptotic normality of the upper-triangular entries of the square root matrix. The resulting limiting distribution resembles that of a scaled Gaussian Wigner ensemble. Additionally, we quantify the rate of convergence using the 1-Wasserstein distance. To validate our theoretical findings, we conduct extensive Monte Carlo simulations, which demonstrate rapid convergence even with relatively low degrees of freedom. These results offer refined insights into the asymptotic behavior of random matrix functionals.

\vspace{0.1in}


% \vspace{0.02in}
% \noindent{\textbf{MSC 2020}}: 49L12, 49N80,	92C20
    \end{abstract}

\section{Introduction}

The Wishart matrix (or ensemble) is a positive-definite random matrix following a Wishart distribution, which arises in a wide range of applications, including multivariate Gaussian modeling, stochastic volatility, random matrix theory, and computational statistics (\cite{And}; \cite{Jon}). The Bartlett decomposition(\cite{Bar}) of a Wishart matrix—typically regarded as a non-symmetric square root—expresses the matrix as the product of a lower-triangular matrix and its transpose. While the Bartlett decomposition has proven to be a powerful tool for developing theoretical properties of the Wishart distribution and for generating random positive-definite matrices (e.g., \cite{Cao}; \cite{Lee}), the symmetric square root—which preserves the eigenbasis of the Wishart matrix—has received comparatively little attention. This paper addresses this gap by analyzing the exact distribution, the asymptotic distribution, and the convergence rate of the upper-triangular entries of the symmetric square root matrix.

% \iffalse

% To situate our work within the existing literature, we briefly review related results on the Wishart distribution and its matrix square roots. The Wishart distribution plays a central role in multivariate statistics, with classical results on its properties established in works such as \cite{And}. In high-dimensional settings, the asymptotics of sample covariance matrices drawn from Wishart laws have been extensively studied (e.g., \cite{Jon}), often focusing on eigenvalues or trace functionals.

% In contrast, the probabilistic structure of the matrix square root of Wishart-distributed matrices—such as those obtained via the Bartlett decomposition—has received relatively little attention. Some studies have examined the Cholesky factors in Bayesian settings or for generating random positive definite matrices (e.g., \cite{Lee} and \cite{Cao}), but a systematic study of the asymptotic behavior of the square root entries, especially their independence and convergence rate, remains lacking.

% Our work addresses this gap by providing a detailed asymptotic analysis of the upper-triangular entries of the matrix square root of a Wishart matrix, demonstrating their convergence to normality and quantifying the convergence rate.

% \fi 

To situate our work within the existing literature, we briefly review related results on the Wishart distribution and its matrix square roots.
The Wishart distribution plays a central role in multivariate statistics, and its fundamental properties have been well established in classical works(\cite{Cao}; \cite{Muirhead}).
It serves as the canonical distribution of sample covariance matrices and has been widely applied in hypothesis testing, estimation, and dimensionality‐reduction methods.

In high‐dimensional settings, the asymptotic behavior of sample covariance matrices drawn from Wishart laws has been extensively studied (\cite{Bai}; \cite{Joh}; \cite{Led}).
Much of this literature focuses on spectral quantities such as eigenvalues, traces, and log‐determinants, which admit elegant limiting distributions under large‐dimension asymptotics.
Beyond eigenvalue analysis, recent studies have also examined fluctuations of matrix functionals and random matrix transformations (\cite{Cou}; \cite{Tao}).

Research on matrix square roots of Wishart matrices has received comparatively less attention, though it appears in areas such as matrix decomposition algorithms, stochastic processes on manifolds, and covariance regularization. 
Notably, the Bartlett decomposition provides a tractable factorization of Wishart matrices that naturally yields expressions for their square roots.
Building upon these foundational results, our work aims to derive the exact and asymptotic distributions of the square root of a standard Wishart matrix and to quantify its convergence properties.

We begin by recalling the definition of the Wishart distribution. The Wishart distribution is a probability distribution over symmetric positive-definite matrices. It can be viewed as a multivariate generalization of the chi-squared distribution and plays a fundamental role in multivariate statistics, particularly in the estimation of covariance matrices.
\begin{definition}
    Let $\boldsymbol{X}_{1}, \boldsymbol{X}_{2}, \ldots, \boldsymbol{X}_{m}$ be independent and identically distributed (i.i.d.) random vectors from a multivariate normal distribution with mean vector $\boldsymbol{0} \in \mathbb{R}^{p}$ and covariance matrix $\Sigma \in \mathbb{R}^{p \times p}>0$, i.e.,
$\boldsymbol{X}_{i} \sim {N}_{p}(\boldsymbol{0}, \Sigma), \quad i=1, \ldots, m$. 
Define  
$$
W =\sum_{i=1}^{m} \boldsymbol{X}_{i}\boldsymbol{X}_i^{\top}.
$$
Then $W$ follows a Wishart distribution with $m$ degrees of freedom and scale matrix $\Sigma$, denoted by
$$W  \sim W_{p}(m, \Sigma).$$

\end{definition}

When $\Sigma = I_p$, the Wishart distribution $W_p(m, I_p)$ is often denoted by $W_p(m)$ and referred to as the standard Wishart distribution. This special case is the primary focus of our study. Throughout this paper, we assume $m \ge p$ to ensure that $W$ is almost surely positive definite.

The symmetric square root of $W$ is defined as the unique positive definite matrix $V$ satisfying $W = V^2$,  denoted by $V = W^{\frac{1}{2}}$.  In many practical applications, it is the symmetric square root $W^{\frac{1}{2}}$, rather than the original Wishart matrix $W$, that serves as the key quantity of interest.  
For instance, the square root of a sample covariance matrix (which is proportional to the Wishart square root) is often used to generate samples with a prescribed covariance structure. 
In wireless communications, the Wishart square root also plays a vital role in optimal precoding for massive MIMO (multiple-input multiple-output) channel modeling.

We then introduce some notation. For any $k\times k$ real-valued matrix $G=(g_{sj})_{1\le s,j\le k}$, 
We   denote $(g_{11},g_{22},\ldots,g_{kk},g_{12},\ldots,g_{1k},\ldots,g_{k-1,k})^\top$ by $G_{up}$, and $(g_{11},g_{22},\ldots,g_{kk}, g_{21},\ldots,g_{k1},\ldots,g_{k,k-1})^\top$ by $G_{lw}$.
The Frobenius norm of matrix $G$ is defined as  
$\|G\|=\sqrt{\sum_{s,j=1}^kg_{sj}^2}$. It is easily verified  that, whenever $G$ is symmetric, $\|G\|^2$ is equal to the square sum of $G$'s eigen values.  Besides, for any two  $k\times k$ real-valued matrix $A$ and $B$, the following inequalities hold:   $\|A+B\|\le \|A\|+\|B\|$ and $\|AB\|\le \|A\|\cdot\|B\|$.

The remainder of this paper is organized as follows.  Section 2 derives the exact distribution of $(W_p(m))^\frac{1}{2}$. 
 Section 3 presents the asymptotic law of $(W_p(m))^\frac{1}{2}$. Section 4 provides an upper bound for the rate of convergence. Section 5 reports the results of numerical experiments. Section 6 offers concluding remarks and further discussion.

\section{Exact distribution of $(W_p(m))^{\frac{1}{2}}$}
In this section, we derive the density function of $(W_p(m))^{\frac{1}{2}}$. 
We consider a more general case, namely the density of $(W_p(m,\Sigma))^{\frac{1}{2}}$ for any positive definite $p\times p$ matrix $\Sigma$.
To achieve this, we first propose the density of $W_p(m,\Sigma)$, as shown in \autoref{density}.
\begin{proposition}\label{density}

    Suppose that $W\sim W_p(m,\Sigma)$. Then, $W$ has the following density function
    \begin{equation*}
        f(W)=\frac{1}{2^{mp/2}|\Sigma|^{m/2}\Gamma_p\left(\frac{m}{2}\right)}|W|^{(m-p-1)/2}\exp\left(-\frac{1}{2}\operatorname{tr}(\Sigma^{-1}W)\right),\quad \forall \  W\in\mathbb{R}^{p\times p},W>0,
    \end{equation*}
    where $\Gamma_p(a) = \pi^{p(p-1)/4} \prod_{i=1}^p\Gamma\left(a-\frac{i-1}{2} \right)
    $ denotes the multivariate Gamma function for $a>(p-1)/2$.
\end{proposition}
A proof of this proposition can be found in \cite{Muirhead}. 
 We will compute the  Jacobian $J(W\to V:=W^{1/2})$ of the matrix square-root transformation,  based on the  result in  \cite{Muirhead}.
 
\begin{proposition}\label{jacobp}
    Let $W$ be a real, symmetric, and positive definite $p \times p$ matrix. Suppose it admits the spectral decomposition
$W = H \Lambda H^\top$, 
where $H \in \mathbb{R}^{p \times p}$ is an orthogonal matrix satisfying $H^\top H = HH^\top = I_p$, and $\Lambda = \text{diag}(\lambda_1, \ldots, \lambda_p)$ with ordered eigenvalues $\lambda_1 \ge \cdots \ge \lambda_p > 0$. 
Then, the Jacobian of the transformation from $W$ to its eigenvalue-eigenvector pair $(\Lambda, H)$ is given by
$$J(W \to (\Lambda, H)) = c_p \prod_{i < j} (\lambda_i - \lambda_j),$$
where
$c_p = \frac{2^p \pi^{p^2/2}}{\Gamma_p(p/2)}$. 
%and $\Gamma_p(\cdot)$ denotes the multivariate Gamma function.
\end{proposition}
From this proposition, we derive the Jacobian $J(W\to V:=W^{1/2})$, as shown in \autoref{jacob}.
\begin{lemma}\label{jacob}
     Suppose that $W_{p\times p}>0$. Then,
     \begin{equation*}
         J(W\to V:=W^{1/2})=\prod_{1\leq i\leq j\leq p}(\delta_i+\delta_j)=2^p\prod_{i=1}^p\delta_i\prod_{i<j}(\delta_i+\delta_j),
     \end{equation*}
     where 
     $\{\delta_j\}_{j=1}^p$ are the eigenvalues of $V$. 
\end{lemma}
\begin{proof}
   Suppose that $W$ admits the spectral decomposition
$W = H \Lambda H^\top$, 
where $H \in \mathbb{R}^{p \times p}$ is an orthogonal matrix satisfying $H^\top H = HH^\top = I_p$, and $\Lambda = \text{diag}(\lambda_1, \ldots, \lambda_p)$ with ordered eigenvalues $\lambda_1\ge \cdots \ge\lambda_p > 0$.   According to \autoref{jacobp}, we have
\begin{equation}\label{eq1}
    J(W\to(\Lambda,H))=c_p\prod_{i<j}(\lambda_i-\lambda_j),\quad c_p:=\frac{2^p\pi^{p^2/2}}{\Gamma_p(p/2)}.
\end{equation}
Let $\Delta:=\text{diag}(\delta_1,\delta_2,\ldots,\delta_p)=\Lambda^{\frac{1}{2}}=\text{diag}(\sqrt{\lambda_1},\sqrt{\lambda_2},\ldots,\sqrt{\lambda_p})$, It follows that
\begin{align}\label{eq2}J(\Lambda\to\Delta)&=\prod_{j=1}^p \left(\frac{d(\delta_j^2)}{d\delta_j}\right)=2^{p}\prod_{j=1}^{p}\delta_{j}.
\end{align}
From the definition of the matrix square root, it follows that $V=H\Lambda^\frac{1}{2}H^\top=H\Delta H^\top$. By applying \autoref{jacobp} once more, we obtain
\begin{align}\label{eq3}
    J((H,\Delta)\to V)&=\frac{1}{J(V\to(H,\Delta)) }
=\frac{1}{c_p\prod_{i<j}(\delta_i-\delta_j)}.
\end{align}
 Combining \eqref{eq1}, \eqref{eq2} and \eqref{eq3}, 
\begin{align*}J(W\to V)&=J(W\to(H,\Lambda))\times J(\Lambda\to\Delta)\times J((H,\Delta)\to V)\nonumber\\&= c_{p}\prod_{i<j}(\delta_{i}^{2}-\delta_{j}^{2})\times(2^{p}\prod_{i=1}^{p}\delta_{i})\times\frac{1}{c_{p}\prod_{i<j}(\delta_{i}-\delta_{j})}\nonumber\\
&=\prod_{i\leq j}(\delta_{i}+\delta_{j}).
\end{align*}
We complete the proof.
\end{proof}
The following theorem presents the main result of this section.
\begin{theorem}\label{density W1/2}
   Suppose that $W \sim W_p(m, \Sigma)$, and let  $V = W^{1/2}$. Then the probability density function of $V$ is given by the following expression,
$$
f(V) = \frac{2^p}{2^{mp/2}\Gamma_p(m/2)|\Sigma|^{m/2}}|V|^{m-p} \exp \left( -\frac{1}{2} \operatorname{tr}(\Sigma^{-1}V^2) \right) \prod_{i < j} (\delta_i + \delta_j),$$
where  $\delta_1, \ldots, \delta_p$ are the eigenvalues of $V$.
\end{theorem}

\begin{proof}
    Let the density function of $W\sim W_p(m,\Sigma)$ be denoted by $p(W)$. From \autoref{density}, we have 
    \begin{equation*}
        p(W)=\frac{1}{2^{mp/2}|\Sigma|^{m/2}\Gamma_p\left(\frac{m}{2}\right)}|W|^{(m-p-1)/2}\exp\left(-\frac{1}{2}\operatorname{tr}(\Sigma^{-1}W)\right),\quad \forall  \  W\in\mathbb{R}^{p\times p},W>0.
    \end{equation*}
    Since $V=W^{\frac{1}{2}}$, the density function of $V$ can be expressed as
    \begin{equation*}
        f(V)=p(V^2)J(W\to V).
    \end{equation*}
  By  applying \autoref{jacob}, we obtain
    \begin{equation*}
        J(W\to V)=\prod_{1\leq i\leq j\leq p}(\delta_{i}+\delta_{j})=2^{p}\prod_{i=1}^{p}\delta_{i}\prod_{i<j}(\delta_{i}+\delta_{j})=2^{p}|V|\prod_{i<j}(\delta_{i}+\delta_{j}).
    \end{equation*}
    Therefore,
    \begin{equation*}
         \begin{aligned}
    f(V)=\frac{1}{2^{mp/2}\Gamma_{p}(m/2)|\Sigma|^{m/2}}|V|^{m-p-1}\exp\left(-\frac{1}{2}tr(\Sigma^{-1}V^{2})\right)\prod_{i\leq j}(\delta_{i}+\delta_{j})\\
    =\frac{2^p}{2^{mp/2}\Gamma_p(m/2)|\Sigma|^{m/2}}|V|^{m-p}\exp\left(-\frac{1}{2}tr(\Sigma^{-1}V^{2})\right)\prod_{i<j}(\delta_{i}+\delta_{j}),
    \end{aligned}
    \end{equation*}
   and we complete the proof.
\end{proof}

\begin{remark}
    From \autoref{density W1/2}, we further know that $V=(W_p(m))^\frac{1}{2}$ has the following density function,
    \begin{equation*}
        \frac{2^p}{2^{mp/2}\Gamma_p(m/2)}|V|^{m-p} \exp \left( -\frac{1}{2} \operatorname{tr}(V^2) \right) \prod_{i < j} (\delta_i + \delta_j),
    \end{equation*}
    where  $\delta_1, \ldots, \delta_p$ denote the eigenvalues of $V$. However, it is challenging to derive the convergence law of 
$V$ directly from this exact density function.
\end{remark}

\section{Asymptotic distribution of $(W_p(m))^{\frac{1}{2}}$}
Although  the explicit form of the density function of $(W_p(m))^\frac{1}{2}$ is available, the distributional behavior of  $(W_p(m))^\frac{1}{2}$ remains largely unexplored. Accordingly, in this section, we investigate 
the asymptotic distribution of the square root of a Wishart matrix $W\sim W_p(m)$ .  

We begin by introducing some additional notation. For two vectors    $\boldsymbol{x}=(x_1,\ldots,x_k)^\top$ and $\boldsymbol{y}=(y_1,\ldots,y_k)^\top$ in $\mathbb{R}^k$, we write $\boldsymbol{x}\le\boldsymbol{y}$  if $x_j\le y_j,\ \forall 1\le j\le k$. Any $\mathbb{R}^k$ valued random vector $\boldsymbol{X}$ is then associated with the distribution function $F_{\boldsymbol{X}}(\boldsymbol{x})=P(\boldsymbol{X}\le \boldsymbol{x} )$. The notions of convergence in distribution and convergence in probability for sequences of random vectors are defined in the following way.

\begin{definition}
    We say that $\boldsymbol{X}_n$ converges in distribution to $\boldsymbol{X}$,denoted as $\boldsymbol{X}_n\xrightarrow{d}\boldsymbol{X}$, if $F_{\boldsymbol{X}_n}(\boldsymbol{x})\to F_{\boldsymbol{X}}(\boldsymbol{x})$ at all continuity points of $F_{\boldsymbol{X}}$ as $n\to \infty$.
\end{definition}
\begin{definition}
     We say that $\boldsymbol{X}_n$ converges in probability to $\boldsymbol{X}$, denoted as $\boldsymbol{X}_n\xrightarrow{P} \boldsymbol{X}$, if for every $\epsilon>0$, 
     $P(\|\boldsymbol{X}_n-\boldsymbol{X}\|>\epsilon)\to 0$, as $\ n\to \infty$.
\end{definition}
Next, we present 
our main result on the asymptotic distribution of $(W_p(m))^\frac{1}{2}$.

\begin{theorem}\label{thmI}
    Fix $ p\in \mathbb{N}^+$ and assume that $W\sim W_p(m, I_p)$, $m\ge p$, and
    \begin{equation*}
        V=W^{\frac 12}=\begin{pmatrix}
            v_{11 }&v_{12}&\cdots&v_{1p}\\
            v_{21 }&v_{22}&\cdots&v_{2p}\\
            \vdots& \vdots& \ddots&\vdots\\
            v_{p1}&v_{p2}&\cdots &v_{pp}
        \end{pmatrix}.
          \end{equation*}
    Then, as $m\to \infty$, the following result holds
    \begin{equation}\label{conres1}
        (v_{11}-\sqrt{m},...,v_{pp}-\sqrt{m},v_{12},...,v_{p-1,p})^{\top} \xrightarrow{d}          \    \mathcal{N}_{p(p+1)/2}\left(\mathbf{0},\left(\begin{array}{cc}\frac{1}{2}I_{p}&0\\0&\frac{1}{4}I_{p(p-1)/2}\end{array}\right)\right).
    \end{equation}
\end{theorem}

\begin{remark}
A $p\times p$ symmetric matrix $G$ is called a Gaussian Wigner matrix if its diagonal entries are i.i.d $N(0,2)$  random variables and its  upper-triangle entries are i.i.d $N(0,1)$ random variables that are independent with the diagonals.  \autoref{thmI} implies 
$$
W^{\frac 12} -\sqrt{m}I_p  \xrightarrow{d} \frac 12 G.
$$
\end{remark}

\begin{remark}
     Since $\sqrt{m-\frac{p}{4}}-\sqrt{m}\rightarrow 0$, it follows that 
     %\eqref{conres1} is equivalently characterized by 
 \begin{equation}\label{conres2}
W^{\frac 12} -\sqrt{m-\frac{p}{4}}I_p  \xrightarrow{d} \frac 12 G,
 \end{equation}
 as a consequence of \autoref{thmI} and \autoref{Slusky}.
Simulation results indicate that \eqref{conres2} provides a numerically closer approximation to the normal distribution.
 \end{remark}

\autoref{thmI} can be established by applying the delta method, based on the known asymptotic distribution of the Wishart matrix (\cite{Muirhead}). In this paper, however, we adopt an alternative approach that more readily facilitates the derivation of convergence rates. Specifically, we employ the Bartlett decomposition of the Wishart matrix (\cite{And}) together with the Slutsky lemma in high-dimensional settings (\cite{Van}), as demonstrated in the following two propositions.

\begin{proposition}\label{Slusky}
    Let $k\in \mathbb{N}^+$. If 
    \begin{equation}\label{cononX}
        \boldsymbol{X}_n  \xrightarrow{d} \boldsymbol{X}\in \mathbb{R}^k,\ n\to\infty,
    \end{equation}
    and 
    \begin{equation}\label{cononY}
        \boldsymbol{Y}_n\xrightarrow{P} \boldsymbol{c}:=(c_1,\ c_2,\ldots,\ c_k)^\top\in \mathbb{R}^k,
    \end{equation}
    where $\boldsymbol{c}$ is a constant vector, then,
    \begin{equation*}
       \boldsymbol{X}_n+  \boldsymbol{Y}_n\xrightarrow{d} \boldsymbol{X}+ \boldsymbol{c}.
    \end{equation*}
\end{proposition}

\begin{proposition}\label{npro}\upshape\textbf{(Bartlett decomposition)}
   \textit{ If the matrix $L_m$ satisfies that}
  \begin{equation}\label{n}
      L_m=\begin{pmatrix}
          c_1&0&0&\ldots&0\\
          n_{21}&c_2&0&\ldots&0\\
          n_{31}&n_{32}&c_3&\ldots&0\\
          \vdots&\vdots& \vdots&\ddots&\vdots\\
          n_{p1}&n_{p2}&n_{p3}&\ldots&c_{p}
      \end{pmatrix},
  \end{equation}
 \textit{ where $c_j^2\sim \chi^2_{m-j+1},\ j=1,2,\ldots,p$, and $n_{ij}\sim \mathcal{N}(0,1), \ 1\le j<i\le p$ are mutually independent, then} \begin{equation*}
      L_mL_m^\top\sim W_p(m).
  \end{equation*}
\end{proposition}

Therefore, the distribution of the square root of  $W_p(m)$ can be equivalently characterized by that of $(L_mL_m^\top)^\frac{1}{2}$. 
To analyze this, we apply the Taylor expansion method to the square root of a positive definite matrix. Before turning to the matrix case, we first consider the one-dimensional case, which leads to the following lemma.

\begin{lemma}\label{lsqrt}
    For any $n\ge2$, we have
    \begin{equation*}
        \sum_{k=0}^{n} \binom{\frac{1}{2}}{k}\binom{\frac{1}{2}}{n-k}=0,
    \end{equation*}
    where
    \begin{equation*}
         \binom{\frac{1}{2}}{k}:=\frac{\frac{1}{2}(\frac{1}{2}-1)\cdots(\frac{1}{2}-k+1)}{k!},\quad \forall k\ge 1;\quad\binom{\frac{1}{2}}{0}:=1.
    \end{equation*}
\end{lemma}
\begin{proof}[Proof. ]
    Applying the Taylor expansion with the Peano form of the remainder yields 
    \begin{equation*}
        (1+x)^\frac{1}{2}=\sum_{k=0}^n\binom{\frac{1}{2}}{k}x^k+o(x^n),\quad (x\to 0).
    \end{equation*}
   Since
    \begin{equation*}
        1+x=\left(\sum_{k=0}^n\binom{\frac{1}{2}}{k}x^k+o(x^n)\right)^2,\quad (x\to 0),
    \end{equation*}
and by the uniqueness of power series expansions, we get
    \begin{equation*}
        \sum_{k=0}^{s} \binom{\frac{1}{2}}{k}\binom{\frac{1}{2}}{s-k}=0,\quad s=2,3,\ldots,n.
    \end{equation*}
    The proof is complete.
\end{proof}

Now, we extend \autoref{lsqrt} to matrix case.

\begin{lemma}\label{psqrt}
    Let $X$ be a $p\times p$ symmetric positive definite matrix. Suppose that $ \|X-I\|=\epsilon\le \frac{1}{2}$, then
    \begin{equation*}
        X^\frac{1}{2}=I+ \frac{1}{2}\left(X-I\right)+R_{X},
    \end{equation*}
    where $R_X$ is a symmetric matrix satisfying
    \begin{equation*}
        \|R_X\|\le 2\epsilon^2.
    \end{equation*}
\end{lemma}

\begin{proof}[Proof.]
    For any $\|X-I\|=\epsilon\le \frac{1}{2}$, we define $(X-I)^0:=I$, and 
    \begin{equation*}
        f_N(X):=\sum_{k=0}^N \binom{\frac{1}{2}}{k}(X-I)^k,\quad \forall N\ge2.
    \end{equation*}
    Note that
    \begin{equation*}
        \left\|\binom{\frac{1}{2}}{k}(X-I)^k\right\|\le \|(X-I)^k\|\le \frac{1}{2^k},\quad \forall  k\ge 1.
    \end{equation*}
    This implies that, 
    \begin{equation*}
        \sup_{l,s\ge N}|(f_l(X)-f_s(X))_{ij}|\le \sup_{l,s\ge N}\|f_l(X)-f_s(X)\|\le \frac{1}{2^{N}},\quad \forall1\le i,j\le p,\ N\ge 2.
    \end{equation*}
    It follows that, for $\forall 1\le i,j\le p$,  $\{(f_N(X))_{ij}\}_{N=2}^\infty$ is a Cauchy sequence in $\mathbb{R}$. Hence, each entry of $f_N(X)$ converges to a finite limit when $N$ goes to $\infty$, implying that 
    $f_N(X)$ converges in $\|\cdot\|$ to a $p\times p$ matrix, denoted by
    \begin{equation*}
        f_\infty(X):=I+\sum_{k=1}^\infty \binom{\frac{1}{2}}{k}(X-I)^k.
    \end{equation*}
   Since each $f_N(X)$ is  symmetric, $ f_\infty(X)$ is also symmetric. We claim that:
    \begin{equation}\label{st1}
         f_\infty^2(X)=X,\quad f_\infty(X)>0 \ (\mbox{positive definite}).
    \end{equation}
  The first statement of \eqref{st1} follows directly from \autoref{lsqrt}. Indeed, we set
    \begin{equation*}
        a_{X,0}=I;\  a_{X,k}=\binom{\frac{1}{2}}{k}(X-I)^k,\ \forall k\ge 1.
    \end{equation*}
    Then, 
    \begin{equation*}
        \sum_{j=0}^\infty \sum_{k=0}^\infty \|a_{X,j}a_{X,k}\|\le
         \sum_{j=0}^\infty \sum_{k=0}^\infty \frac{p}{2^{j+k}}<\infty.
    \end{equation*}
    Therefore, changing the order of summation does not affect the value of $\sum_{j=0}^\infty \sum_{k=0}^\infty a_{X,j}a_{X,k}$. This leads to
    \begin{equation*}
        f_\infty^2(X)=\left(\sum_{j=0}^\infty a_{X,j}\right)^2= \sum_{j=0}^\infty \sum_{k=0}^\infty a_{X,j}a_{X,k}=\sum_{s=0}^\infty \left(\sum_{j=0}^sa_{X,j}a_{X,s-j}\right)=I+X-I=X,
    \end{equation*}
    where we have applied results established in  \autoref{lsqrt}. 
    Next, we verify the second statement of \eqref{st1}. For any $p$-dimensional vector $\boldsymbol{x}$ satisfying $\|\boldsymbol{x}\|=1$, we have
    \begin{equation*}
        \boldsymbol{x}^\top f_\infty(X)\boldsymbol{x}\ge 1-\sum_{k=1}^\infty\left|\binom{\frac{1}{2}}{k}\right|\|X-I\|^k\ge 1-\frac{1}{4}-\sum_{k=2}^\infty \frac{1}{2^k}>0,
    \end{equation*}
   which implies that $ f_{\infty}(X) > 0 $. Therefore, \eqref{st1} holds, and consequently, 
    \begin{equation*}
        X^\frac{1}{2}= f_\infty(X)=I+\sum_{k=1}^\infty \binom{\frac{1}{2}}{k}(X-I)^k.
    \end{equation*}
    Lastly, from the estimation below,
    \begin{equation*}
     \sum_{k=2}^\infty \left\|  \binom{\frac{1}{2}}{k}(X-I)^k\right\|\le 
      \left\|(X-I)^2\right\|\left(\sum_{k=2}^\infty \frac{1}{2^{k-2}}\right)\le 2\epsilon^2,
    \end{equation*}
    we obtain the desired result.
\end{proof}
 The next lemma characterize the limiting distribution  of the $\chi$ random variable (i.e. the square root of a chi-squared variable), which corresponds to the case $p=1$  in   \autoref{thmI}.

\begin{lemma}\label{lodchi}
For any fixed $x\in \mathbb{Z}$, we have
    \begin{equation}\label{odchi}
        \sqrt{\chi_{m+x}^2}-\sqrt{m}\overset{d}{\to} \mathcal{N}\left(0,\frac{1}{2}\right),\quad m\to\infty,
    \end{equation}
where, throughout, $\chi^2_{m+x}$ denotes a random variable following a $\chi_{m+x}^2$ distribution.
\end{lemma}
\begin{proof}[Proof.]
    From the Central Limit Theorem, it follows that
    \begin{equation*}
        \frac{\chi_{m+x}^2-(m+x)}{\sqrt{m+x}}\overset{d}{\to} \mathcal{N}(0,2),\quad m\to \infty.
    \end{equation*}
By the law of large numbers, we have 
    \begin{equation*}
       \frac{{\chi_{m+x}^2}}{{m+x}}\xrightarrow{P } 1,\quad 
       \frac{\sqrt{m+x}}{\sqrt{\chi_{m+x}^2}+\sqrt{m+x}}\xrightarrow{P }  \frac{1}{2},\quad m\to\infty.
    \end{equation*}
    and thus, by Slutsky's lemma
    \begin{equation}\label{d}
        \sqrt{\chi_{m+x}^2}-\sqrt{m+x}=\frac{\chi_{m+x}^2-{(m+x)}}{\sqrt{\chi^2_{m+x}}+\sqrt{m+x}}\overset{d}{\to} \mathcal{N}\left(0,\frac{1}{2}
        \right).
    \end{equation}
Furthermore, noting that
    \begin{equation}\label{p}
        \sqrt{m+x}-\sqrt{m} \rightarrow  0,\quad m\to \infty,
    \end{equation}
and applying  Slutsky's lemma once again, we obtain    \eqref{odchi}  from \eqref{d} and \eqref{p}. 
    Hence, the proof is complete.
\end{proof}
We now turn to the proof of \autoref{thmI}.
\begin{proof}[ Proof of \autoref{thmI}]\label{pf2}
Define $L_m$ as in \eqref{n}, and 
set
\begin{equation*}
  T_m:= L_m-\sqrt{m}I.
\end{equation*}
Then, $T_m$ is a lower triangular matrix whose diagonal entries satisfy the following distributional properties:
\begin{equation*}
    (T_m)_{jj}\sim \sqrt{\chi^2_{m-j+1}}-\sqrt{m} ,\quad j=1,2,\ldots,p,
\end{equation*}
and whose strictly lower-triangular entries follow $\mathcal{N}(0,1)$.
Since the entries of $(L_m)_{lw}$ are independent, it follows that the components of $(T_m)_{lw}$ are also mutually independent. Now, we  apply \autoref{psqrt} to characterize the square root of $L_mL_m^\top$. Note that if $\|T_m\|\le \frac{\sqrt{m}}{5}$, then
\begin{equation*}
    \left\|\frac{T_m+T_m^\top}{\sqrt{m}}+\frac{T_mT_m^\top}{m}\right\|\le \frac{2\|T_m\|}{\sqrt{m}}+\frac{\|T_m\|^2}{m} <\frac{1}{2}.
\end{equation*}
Therefore, by \autoref{psqrt}, we have
\begin{equation*}
    \left(I+\frac{T_m+T_m^\top}{\sqrt{m}}+\frac{T_mT_m^\top}{m}\right)^\frac{1}{2}=I+\frac{T_m+T_m^\top}{2\sqrt{m}}+\frac{T_mT_m^\top}{2m}+R,
\end{equation*}
where $R$ is symmetric and satisfies
\begin{equation*}
    \|R\|\le 2 \left(\frac{11}{5}\frac{\|T_m\|}{\sqrt{m}}\right)^2<10\left(\frac{\|T_m\|}{\sqrt{m}}\right)^2.
\end{equation*}
Since $\|\frac{T_mT_m^\top}{m}\|\le \left(\frac{\|T_m\|}{\sqrt{m}}\right)^2$,  absorbing the term  $\frac{T_mT_m^\top}{2m}$  into the remainder yields
\begin{equation*}
     \left(I+\frac{T_m+T_m^\top}{\sqrt{m}}+\frac{T_mT_m^\top}{m}\right)^\frac{1}{2}=I+\frac{T_m+T_m^\top}{2\sqrt{m}}+R',\quad\|R'\|\le 11\left(\frac{\|T_m\|}{\sqrt{m}}\right)^2,\quad \text{if }\frac{\|T_m\|}{\sqrt{m}}\le\frac{1}{5}.
\end{equation*}
Direct calculation further shows
\begin{align}\label{del}
    (L_mL_m^\top)^\frac{1}{2}-\sqrt{m}I=&\left((\sqrt{m}I+T_m)(\sqrt{m}I+T_m^\top)\right)^\frac{1}{2}-\sqrt{m}I\nonumber       \\
    =& \sqrt{m}\left(I+\frac{T_m+T_m^\top}{\sqrt{m}}+\frac{T_mT_m^\top}{m}\right)^\frac{1}{2}-\sqrt{m}I\nonumber\\
    =&1_{\|T_m\|\le \frac{\sqrt{m}}{5}}\left(\left(\frac{T_m+T_m^\top}{2}\right)+\sqrt{m}R'\right)+1_{\|T_m\|> \frac{\sqrt{m}}{5}}\left( (L_mL_m^\top)^\frac{1}{2}-\sqrt{m}I\right)\nonumber\\
    :=& \frac{T_m+T_m^\top}{2}+1_{\|T_m\|\le \frac{\sqrt{m}}{5}}\sqrt{m}R'+1_{\|T_m\|> \frac{\sqrt{m}}{5}}R'',
\end{align}
where
\begin{equation*}
    R''=1_{\|T_m\|> \frac{\sqrt{m}}{5}}\left( \left(L_mL_m^\top\right)^\frac{1}{2}-\sqrt{m}I-\frac{T_m+T_m^\top}{2}\right).
\end{equation*}
Next, we evaluate $\mathbb{E}[\|T_m\|^2]$. For any $j=1,2,\ldots,p$, we have
$$\mathbb{E}\left[\left(\chi^2_{m-j+1}-m\right)^2\right] =2(m-j+1)+(j-1)^2,
$$
which implies that
\begin{equation*}
    \mathbb{E}\left(\sqrt{\chi^2_{m-j+1}}-\sqrt{m}\right)^2
      =\mathbb{E}\left( \frac{\chi^2_{m-j+1}-m}{\sqrt{\chi^2_{m-j+1}}+\sqrt{m}} \right)^2
      \le \mathbb{E} \left( \frac{\chi^2_{m-j+1}-m}{ \sqrt{m}} \right)^2 
    = \frac{2(m-j+1)+(j-1)^2}{m}.
\end{equation*}
Therefore,
\begin{align*}
    \mathbb{E}[\|T_m\|^2]&\le \frac{p(p-1)}{2}+\sum_{j=1}^p \frac{2(m-j+1)+(j-1)^2}{m}\\
    &<\frac{p^3}{3m}+\frac{p(p+3)}{2}\\
    &\le \frac{5}{6}p^2+\frac{3}{2}p.
\end{align*}
It follows from Chebyshev's Inequality that, as $m\to \infty$
\begin{equation}\label{p1}
    P\left(\|T_m\|>\frac{\sqrt{m}}{5}\right)\le \frac{25}{m}\left( \frac{5}{6}p^2+\frac{3}{2}p\right)\to 0, \quad 1_{\|T_m\|> \frac{\sqrt{m}}{5}}R''_{up}\overset{P}{\to} \boldsymbol{0}.
\end{equation}
In addition,
\begin{equation*}
    \mathbb{E}\left[1_{\|T_m\|\le \frac{\sqrt{m}}{5}}\sqrt{m}\|R'\|\right]\le 11\mathbb{E}\left[\frac{\|T_m\|^2}{\sqrt{m}}\right]\le \frac{11}{\sqrt{m}}\left(\frac{5}{6}p^2+\frac{3}{2}p\right)\to 0,\quad m\to \infty.
\end{equation*}
Since $\|R_{up}'\|\le \|R'\|$, Chebyshev's Inequality further implies that
\begin{equation}\label{p2}
    1_{\|T_m\|\le \frac{\sqrt{m}}{5}}\sqrt{m}R'_{up}\overset{P}{\to }\boldsymbol{0}.
\end{equation}
Note that the entries of $\left( \frac{T_m+T_m^\top}{2}\right)_{up}$ are independently distributed. Specifically, the first $p$ entries, which follow $\sqrt{\chi^2_{m-j+1}}-\sqrt{m}$ for $\ j=1,2,\ldots,p$, respectively, converge in  distribution to $\mathcal{N}(0,\frac{1}{2})$, as established in  \autoref{lodchi}. Moreover, the remaining entries follow $\mathcal{N}(0,\frac{1}{4})$. Hence, we conclude that,

\begin{equation}\label{p3}
    \left( \frac{T_m+T_m^\top}{2}\right)_{up}\overset{d}{\to}  \mathcal{N}_{p(p+1)/2}\left(\mathbf{0},\left(\begin{array}{cc}\frac{1}{2}I_{p}&0\\0&\frac{1}{4}I_{p(p-1)/2}\end{array}\right)\right).
\end{equation}
By selecting the upper triangular entries of \eqref{del}, and combining  \eqref{p1}, \eqref{p2} and \eqref{p3}, we obtain
\begin{equation*}
    \left( (L_mL_m^\top)^\frac{1}{2}-\sqrt{m}I\right)_{up}\overset{d}{\to } 
    \mathcal{N}_{p(p+1)/2}\left(\mathbf{0},\left(\begin{array}{cc}\frac{1}{2}I_{p}&0\\0&\frac{1}{4}I_{p(p-1)/2}\end{array}\right)\right),
\end{equation*}
where we have applied Slutsky’s lemma (\autoref{Slusky}).
Finally, from \autoref{npro}, we arrive at \eqref{conres1}. The proof is complete.
\end{proof}

\section{Convergence rate of $(W_p(m))^{\frac{1}{2}}$}

In this section, we evaluate the convergence rate of $W^{\frac{1}{2}}$. 
Specifically, we elaborate on the proof of \autoref{thmI} and derive an upper bound on the convergence rate in terms of the 1-Wasserstein distance, defined as follows.
\begin{definition}
    For $\mathbb{R}^d$-valued random vectors $\boldsymbol{X}$
and $\boldsymbol{Y}$, the 1-Wasserstein distance between them is denoted by 
\begin{equation*}
    d_W(\boldsymbol{X},\boldsymbol{Y}):=\sup_{h:Lip(h)\le 1}  \left|\mathbb{E}[h(\boldsymbol{X})-h(\boldsymbol{Y})]\right|.
\end{equation*}
\end{definition}
Some useful properties of this distance are summarized below (we assume that $\boldsymbol{X},\ \boldsymbol{Y},\ \boldsymbol{Z},\ \boldsymbol{T}$ are $\mathbb{R}^d$-valued random vectors for some $d\in \mathbb{N}^+$):
\begin{align*}
   (P_1)&:\  d_W(a\boldsymbol{X},a\boldsymbol{Y})=|a|d_W(\boldsymbol{X},\boldsymbol{Y}),\ d_W(\boldsymbol{x}+\boldsymbol{X},\boldsymbol{x}+\boldsymbol{Y})=d_W(\boldsymbol{X},\boldsymbol{Y}),\ \forall a\in \mathbb{R},\boldsymbol{x}\in \mathbb{R}^d.\\
   (P_2)&: \  d_W(\boldsymbol{X},\boldsymbol{Y})\le d_W(\boldsymbol{X},\boldsymbol{Z})+d_W(\boldsymbol{Z},\boldsymbol{Y}).        \\
   (P_3)&:\  d_W(\boldsymbol{X}+\boldsymbol{Z},\boldsymbol{Y}+\boldsymbol{Z})=\sup_{h:Lip(h)\le 1}  |\mathbb{E}[h(\boldsymbol{X}+\boldsymbol{Z})-h(\boldsymbol{Y}+\boldsymbol{Z})]| \le \mathbb{E}\left[\|\boldsymbol{X}-\boldsymbol{Y}\|\right].                      \\
   (P_4)&:\   d_W(\boldsymbol{X}+\boldsymbol{Y},\boldsymbol{Z})\le d_W(\boldsymbol{X},\boldsymbol{Z})+d_W(\boldsymbol{X},\boldsymbol{X}+\boldsymbol{Y})\le d_W(\boldsymbol{X},\boldsymbol{Z})+\mathbb{E}\left[\|\boldsymbol{Y}\|\right].                   \\
   (P_5)&:\ \text{If $\boldsymbol{X}$, $\boldsymbol{Y}$,  $\boldsymbol{Z}$ and $\boldsymbol{T}$ are independent, then } d_W(\boldsymbol{X}+\boldsymbol{Y},\boldsymbol{Z}+\boldsymbol{T})\le d_W(\boldsymbol{X},\boldsymbol{Z})+d_W(\boldsymbol{Y},\boldsymbol{T}).\\
   (P_6)&:\  \text{If } X_1,X_2,\ldots,X_k \text{ and } Y_1,Y_2,\ldots,Y_k \text{ are independent one-dimensional random variables, then }\\
   &\quad d_W((X_1,X_2,\ldots,X_k)^\top, \ (Y_1,Y_2,\ldots,Y_k)^\top)\le \sum_{j=1}^k d_W(X_j,Y_j).
\end{align*}
Properties $(P_1)\verb| - |(P_4)$ can be easily verified. Property $(P_5)$ holds  because
\begin{align*}
    d_W(\boldsymbol{X}+\boldsymbol{Y},\boldsymbol{Z}+\boldsymbol{T})&\le d_W(\boldsymbol{X}+\boldsymbol{Y},\boldsymbol{Z}+\boldsymbol{Y}) +d_W(\boldsymbol{Z}+\boldsymbol{Y},\boldsymbol{Z}+\boldsymbol{T})  \\
    &= \sup_{h:Lip(h)\le 1} \left|\mathbb{E}\left[\mathbb{E}[h(\boldsymbol{X}+\boldsymbol{Y})-h(\boldsymbol{Z}+\boldsymbol{Y})|\boldsymbol{Y}]\right]\right|\\
    &\quad+\sup_{h:Lip(h)\le 1} |\mathbb{E}\left[\mathbb{E}[h(\boldsymbol{Z}+\boldsymbol{Y})-h(\boldsymbol{Z}+\boldsymbol{T})|\boldsymbol{Z}]\right]|\\
    &\le d_W(\boldsymbol{X},\boldsymbol{Z})+d_W(\boldsymbol{Y},\boldsymbol{T}).
\end{align*}
In the final step, we utilize the independence property, which ensures that the conditional distributions of  $(\boldsymbol{X},\boldsymbol{Z})|_{\boldsymbol{Y}}$ and $( \boldsymbol{Y},\boldsymbol{T})|_{\boldsymbol{Z}}$ coincide with those of  $(\boldsymbol{X},\boldsymbol{Z})$ and $( \boldsymbol{Y},\boldsymbol{T})$, respectively.
To prove $(P_6)$, note that the sets $\{(0,\ldots,X_j,\ldots,0)\}_{1\le j\le k}$ and 
$\{(0,\ldots,Y_j,\ldots,0)\}_{1\le j\le k}$ are independent The result then  follows by  induction together with  $(P_5)$.

Now, following the approach used in the proof of \autoref{thmI}, we derive an upper bound for the convergence rate of $W^\frac{1}{2}$ in the 1-Wasserstein distance.

\begin{theorem}\label{secbound}
    Suppose that $m\ge 2p$, and
    \begin{equation*}
        \boldsymbol{Z'}\sim \mathcal{N}_{p(p+1)/2}\left(\mathbf{0},\left(\begin{array}{cc}\frac{1}{2}I_{p}&0\\0&\frac{1}{4}I_{p(p-1)/2}\end{array}\right)\right).
    \end{equation*}
    Then,
    \begin{equation*}
        d_W\left(\left(\left(W_p(m)\right)^\frac{1}{2}-\sqrt{m}I\right)_{up},\boldsymbol{Z'}\right)< \frac{1}{\sqrt{m}}\left(42p^\frac{5}{2}+9p^\frac{9}{4}+18p^2+75p^\frac{3}{2}+15p^\frac{5}{4}+56p\right)=O(p^{2.5}/\sqrt{m}).
    \end{equation*}
\end{theorem}

The following lemma provides a bound in the one-dimensional case.

\begin{lemma}\label{x1}
     Assume that $X_1\sim \chi^2_m$. Then,
    \begin{equation*}
        d_W\left(\sqrt{X_1}-\sqrt{m},\mathcal{N}\left(0,\frac{1}{2}\right)\right)< \frac{16}{\sqrt{m}}.
    \end{equation*}
\end{lemma}
\begin{proof}[Proof.]
    By applying \autoref{psqrt} to $\frac{X_1}{m}$, we obtain
    \begin{equation*}
        \sqrt{X_1}-\sqrt{m}=\frac{X_1-m}{2\sqrt{m}}+R_1,
    \end{equation*}
    where
    \begin{align*}
        &|R_1|\le 2\sqrt{m}\left(\frac{X_1-m}{m}\right)^2,\ \text{if $\left|\frac{X_1-m}{m}\right|\le \frac{1}{2}$},\\
         &|R_1|= \left|\sqrt{X_1}-\sqrt{m}-\frac{X_1-m}{2\sqrt{m}}\right|<\frac{X_1}{2\sqrt{m}},\ \text{if $\frac{X_1}{m}> \frac{3}{2}$},\\
         &|R_1|= \left|\sqrt{X_1}-\sqrt{m}-\frac{X_1-m}{2\sqrt{m}}\right|\le\frac{\sqrt{m}}{2},\ \text{if $\frac{X_1}{m}< \frac{1}{2}$}.
    \end{align*}
   We present different estimates for three distinct cases, as shown below:
    \begin{align}\label{thres}
        \mathbb{E}\left[1_{\frac{X_1}{m}> \frac{3}{2}}|R_1|\right]
        &<\mathbb{E}\left[\frac{X_1}{2\sqrt{m}}1_{\frac{X_1}{m}> \frac{3}{2}}\right]\nonumber\\ 
        &= \int_{\frac{3}{2}m}^\infty \frac{1}{2\sqrt{m}}P\left(X_1\ge x\right)dx+
        \frac{3\sqrt{m}}{4}P\left(X_1-m>\frac{m}{2}\right)\nonumber\\
       & \le  \frac{1}{2\sqrt{m}} \int_{\frac{3}{2}m}^\infty \frac{\mathbb{E}[(X_1-m)^2]}{(x-m)^2}dx+
        \frac{3\sqrt{m}}{4}P\left(X_1-m>\frac{m}{2}\right)\nonumber\\
        &= \frac{2}{\sqrt{m}}+\frac{3\sqrt{m}}{4}P\left(X_1-m>\frac{m}{2}\right).
        \end{align}
        In addition,
        \begin{align}\label{thres2}
         &\mathbb{E}\left[1_{\left|\frac{X_1-m}{m}\right|\le \frac{1}{2}}|R_1|\right]\le \mathbb{E}\left[ \frac{2\sqrt{m}(X_1-m)^2}{m^2}\right]=\frac{4}{\sqrt{m}},\nonumber\\
          &\mathbb{E}\left[1_{\frac{X_1}{m}< \frac{1}{2}}|R_1|\right]\le \frac{\sqrt{m}}{2}P\left(X_1<\frac{m}{2}\right)\le\frac{\sqrt{m}}{2}P\left(X_1-m\le -\frac{m}{2}\right). 
    \end{align} 
   
    % \begin{equation*}
    %     \mathbb{E}[1_{\|\frac{X-m}{m}\|\le \frac{1}{2}}R]\le \mathbb{E}[2\sqrt{m} \frac{(X-m)^2}{m^2}]=\frac{4}{\sqrt{m}}.
    % \end{equation*}
    
 Chebyshev's inequality implies that
\begin{equation}\label{ch}
    P\left(X_1-m>\frac{m}{2}\right)+P\left(X_1-m<-\frac{m}{2}\right)\le \frac{\mathbb{E}[(X_1-m)^2]}{\frac{m^2}{4}}=\frac{8}{m}.
\end{equation}
\eqref{thres}, \eqref{thres2} and \eqref{ch} together yield
\begin{equation}\label{com}
    \mathbb{E}\left[1_{\frac{X_1}{m}> \frac{3}{2}}|R_1|\right]+\mathbb{E}\left[1_{\frac{X_1}{m}< \frac{1}{2}}|R_1|\right]\le
    \frac{8}{\sqrt
    {m}}.
\end{equation}
 From the Wasserstein distance bound in \cite{Rei}, we have
    \begin{align*}\label{bs}
        d_W\left(\frac{X_1-m}{2\sqrt{m}},\mathcal{N}\left(0,\frac{1}{2}\right)\right)&=\frac{1}{\sqrt{2}}d_W\left(\frac{\chi_m^2-m}{\sqrt{2m}},\mathcal{N}\left(0,1\right)\right)\\
        &\le  \frac{1}{\sqrt{2m}}  \left(2+\mathbb{E}\left[\left|\frac{\chi_1^2-1}{\sqrt{2}}\right|^3\right]\right)   \\
        &\le \frac{5.1}{\sqrt{2m}}.
   \end{align*}
    By combining \eqref{thres}, \eqref{thres2} and \eqref{com}, and applying Property $(P_4)$, we obtain
    \begin{align*}
        d_W\left(\sqrt{X_1}-\sqrt{m},\mathcal{N}\left(0,\frac{1}{2}\right)\right)\le& d_W\left(\frac{X_1-m}{2\sqrt{m}},\mathcal{N}\left(0,\frac{1}{2}\right)\right)+\mathbb{E}[|R_1|]\\
        \le& d_W\left(\frac{X_1-m}{2\sqrt{m}},\mathcal{N}\left(0,\frac{1}{2}\right)\right)+\mathbb{E}[1_{\frac{X_1}{m}> \frac{3}{2}}|R_1|]\\&+
\mathbb{E}[1_{\left|\frac{X_1-m}{m}\right|\le \frac{1}{2}}|R_1|]+
\mathbb{E}\left[1_{\left|\frac{X_1}{m}\right|< \frac{1}{2}}|R_1|\right]\\
        < &\frac{16}{\sqrt{m}}.
    \end{align*}
    The proof is complete. 
\end{proof}
\begin{proof}[Proof of \autoref{secbound}]    We now adopt the same notation as in the proof of \autoref{thmI}.
    Recall that $T_m$ is a lower triangular matrix with independent entries, whose diagonal elements satisfy
    $(T_m)_{jj}\sim \sqrt{\chi^2_{m-j+1}}-\sqrt{m} ,\ \forall j=1,2,\ldots,p,$
while the strictly lower-triangular entries follow $\mathcal{N}(0,1)$. Recall also that
\begin{equation*}
    \mathbb{E}\left[\|T_m\|^2\right]\le \frac{5}{6}p^2+\frac{3}{2}p,
\end{equation*}
and 
 \begin{align*}
    (L_mL_m^\top)^\frac{1}{2}-\sqrt{m}I= \frac{T_m+T_m^\top}{2}+1_{\|T_m\|\le \frac{\sqrt{m}}{5}}\sqrt{m}R'+1_{\|T_m\|> \frac{\sqrt{m}}{5}}R''.
\end{align*}
We have already obtained an estimate for $\mathbb{E}\left[1_{\|T_m\|\le \frac{\sqrt{m}}{5}}\sqrt{m}\|R'_{up}\|\right]$ in the proof of \autoref{thmI}, 
    \begin{align*}
    \mathbb{E}\left[1_{\|T_m\|\le \frac{\sqrt{m}}{5}}\sqrt{m}\|R'_{up}\|\right]\le  \frac{11}{\sqrt{m}}\left(\frac{5}{6}p^2+\frac{3}{2}p\right).
\end{align*}

Now, we proceed to evaluate $\mathbb{E}[1_{\|T_m\|> \frac{\sqrt{m}}{5}}\|R''_{up}\|]$. Note that if ${\|T_m\|> \frac{\sqrt{m}}{5}}$, then
    \begin{align*}
        \|R''\|&=  \left\|\sqrt{m}\left(I+\frac{T_m+T_m^\top}{\sqrt{m}}+\frac{T_mT_m^\top}{m}\right)^\frac{1}{2}-\sqrt{m}I-\frac{T_m+T_m^\top}{2}\right\|\\
        &\le \sqrt{m}p^{\frac{1}{4}}\left(\frac{\|T_m\|^2}{m}+\frac{2\|T_m\|}{\sqrt{m}}+\sqrt{p}\right)^\frac{1}{2}+\sqrt{mp}+\|T_m\|\\
        &\le p^\frac{1}{4}{\|T_m\|}+2\sqrt{mp}+\|T_m\|.
    \end{align*}
In the first inequality, we use the following fact: if $G$ is a $p\times p$ symmetric positive definite matrix, then $\|\sqrt{G}\|^2\le \sqrt{p} \|G\|$. This inequality follows directly from the Cauchy–Schwarz inequality. We then compute  
    \begin{align*}
   \mathbb{E}\left[1_{\|T_m\|> \frac{\sqrt{m}}{5}}\|T_m\|\right]&= \int_\frac{\sqrt{m}}{5}^\infty P\left(\|T_m\|\ge x\right)dx+\frac{\sqrt{m}}{5}P\left(\|T_m\|> \frac{\sqrt{m}}{5}\right)\\
   &\le\int_\frac{\sqrt{m}}{5}^\infty \frac{\mathbb{E}[\|T_m\|^2]}{x^2}dx+\frac{5}{\sqrt{m}}\left(\frac{5}{6}p^2+\frac{3}{2}p\right)\\
   &\le \frac{10}{\sqrt{m}}\left(\frac{5}{6}p^2+\frac{3}{2}p\right),
    \end{align*}
    \begin{equation*}
        \mathbb{E}\left[1_{\|T_m\|> \frac{\sqrt{m}}{5}}\sqrt{m}\right]=\sqrt{m}P\left(\|T_m\|\ge \frac{\sqrt{m}}{5}\right)\le \frac{25}{\sqrt{m}}\left(\frac{5}{6}p^2+\frac{3}{2}p\right).
    \end{equation*}
    Therefore,
    \begin{align*}
        \mathbb{E}\left[1_{\|T_m\|> \frac{\sqrt{m}}{5}}\|R''_{up}\|\right]&\le \mathbb{E}\left[1_{\|T_m\|> \frac{\sqrt{m}}{5}}\|R''\|\right]\\
        &\le \frac{10}{\sqrt{m}}\left(p^\frac{1}{4}+1\right)\left(\frac{5}{6}p^2+\frac{3}{2}p\right)+\frac{50\sqrt{p}}{\sqrt{m}} \left(\frac{5}{6}p^2+\frac{3}{2}p\right).
    \end{align*}
      From Property $(P_6)$, we have
    \begin{align}\label{main}
        d_W\left( \left(\frac{T_m+T_m^\top}{2}\right)_{up},\boldsymbol{Z}'\right)\le &\sum_{j=1}^p d_W\left(\sqrt{\chi^2_{m-j+1}}-\sqrt{m},\mathcal{N}\left(0,\frac{1}{2}\right)\right)\nonumber
        \\
        \le &\sum_{j=1}^p \left(\left(\sqrt{m}-\sqrt{m-j+1}\right)+\frac{16}{\sqrt{m-j+1}}\right)\nonumber\\
        <& \frac{24p}{\sqrt{m}}+\frac{p^2}{2\sqrt{2m}},
    \end{align}
 where \autoref{x1} has been applied. Combining \eqref{del} and \eqref{main}, we conclude that
    \begin{align*}
        d_W\left(\left((L_mL_m^\top)^\frac{1}{2}-\sqrt{m}I\right)_{up},Z'\right)&\le d_W\left( \left(\frac{T_m+T_m^\top}{2}\right)_{up},\boldsymbol{Z}'\right)\\&\quad+ \mathbb{E}\left[1_{\|T_m\|\le \frac{\sqrt{m}}{5}}\sqrt{m}\|R'_{up}\|+1_{\|T_m\|> \frac{\sqrt{m}}{5}}\|R''_{up}\|\right]\\
        &\le\frac{24p}{\sqrt{m}}+\frac{p^2}{2\sqrt{2m}}+\frac{11}{\sqrt{m}}\left(\frac{5}{6}p^2+\frac{3}{2}p\right)\\
        &\quad+ \frac{10}{\sqrt{m}}\left(p^\frac{1}{4}+1\right)\left(\frac{5}{6}p^2+\frac{3}{2}p\right)+\frac{50\sqrt{p}}{\sqrt{m}} \left(\frac{5}{6}p^2+\frac{3}{2}p\right)\\
        &< \frac{1}{\sqrt{m}}\left(42p^\frac{5}{2}+9p^\frac{9}{4}+18p^2+75p^\frac{3}{2}+15p^\frac{5}{4}+56p\right),
    \end{align*}
 and the desired result follows from \autoref{npro}.
\end{proof}

\section{Numerical experiments}
In this section, we present a simulation study for $V:=(W_p(m))^\frac{1}{2}$. Specifically, we generate $N=10000$ samples from $W_p(m)$ and 
 and compute their corresponding square root matrices via spectral decomposition. For 
$p=2$ and 
$m=5$, the empirical distribution of 
$V_{11}-\sqrt{m-\frac{p}{4}}$
  closely matches a 
$\mathcal{N}(0,\frac{1}{2})$ distribution, while the off-diagonal entry 
$V_{12}$
  aligns with 
$\mathcal{N}(0,\frac{1}{4})$, as illustrated in \autoref{52,11}. The Q–Q plots shown in \autoref{52,QQ} further support the asymptotic normality of the entries.

The empirical mean values and the  covariance matrix of the entries of 
$V$
  for 
$m=5,\  p=2$ are reported in Table \ref{tab:mean} and \ref{tab:covariance}, respectively. The mean values of the diagonal entries are close to $\sqrt{m-\frac{p}{4}}\approx2.12$, while the mean of the off-diagonal  entry is close to 0. The empirical variances of the diagonal entries are approximately 
0.5, and that of the off-diagonal entry is around 
0.25
, consistent with the theoretical limiting variances. Furthermore, the covariances of different entries are small, confirming the asymptotic independence predicted by theory.

These results indicate that the proposed approximation performs well even in low-dimensional settings, and remains accurate when 
$m$ is only moderately larger than 
$p$.

  \begin{figure}[H]
    \centering
    \begin{subfigure}[t]{0.48\linewidth}
        \centering
        \includegraphics[width=\linewidth]{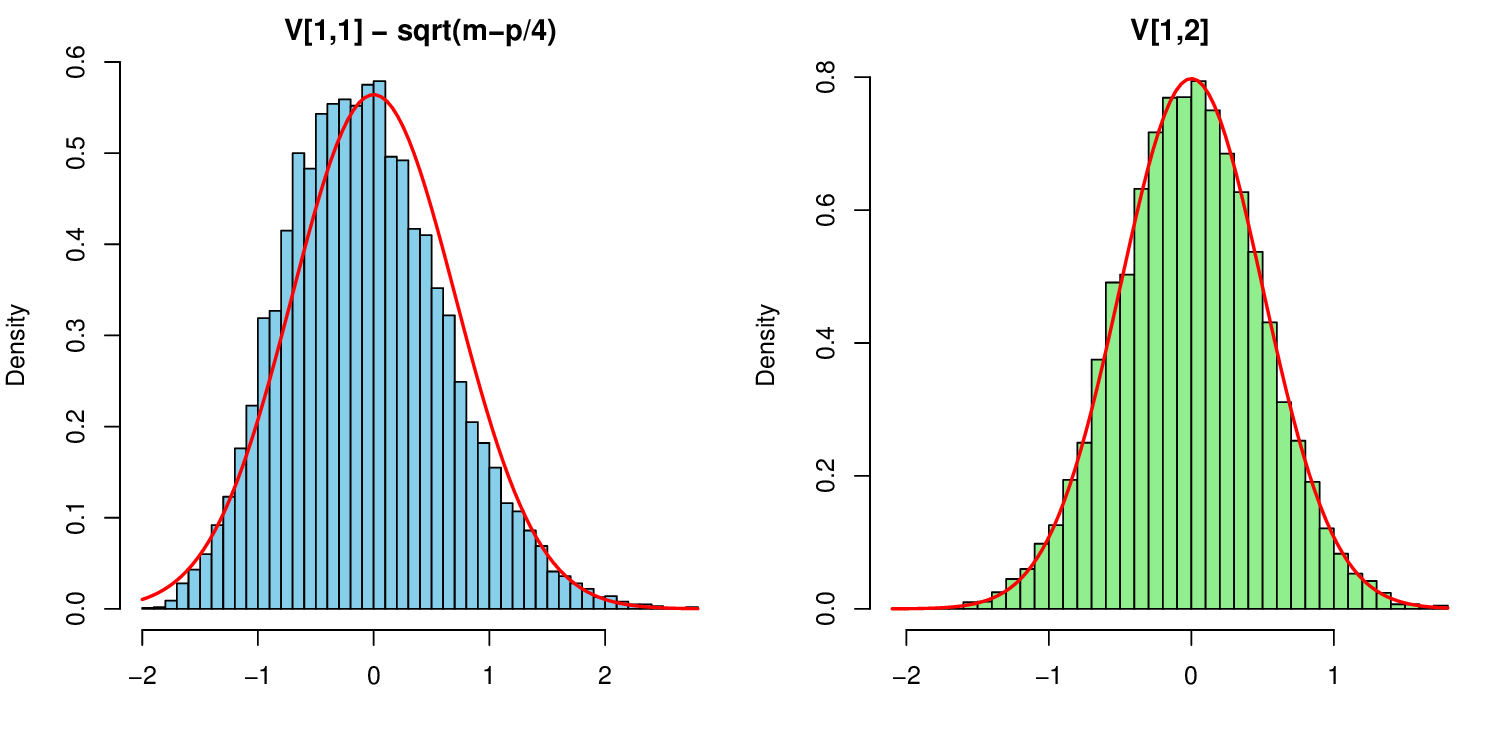}
        \caption{Comparison of empirical densities  of \( V_{11} - \sqrt{m-\frac{p}{4}} \) and \( V_{12} \), overlaid with \( \mathcal{N}(0, 1/2) \) and \( \mathcal{N}(0, 1/4) \), respectively}
        \label{52,11}
    \end{subfigure}
    \hfill
    \begin{subfigure}[t]{0.48\linewidth}
        \centering
        \includegraphics[width=\linewidth]{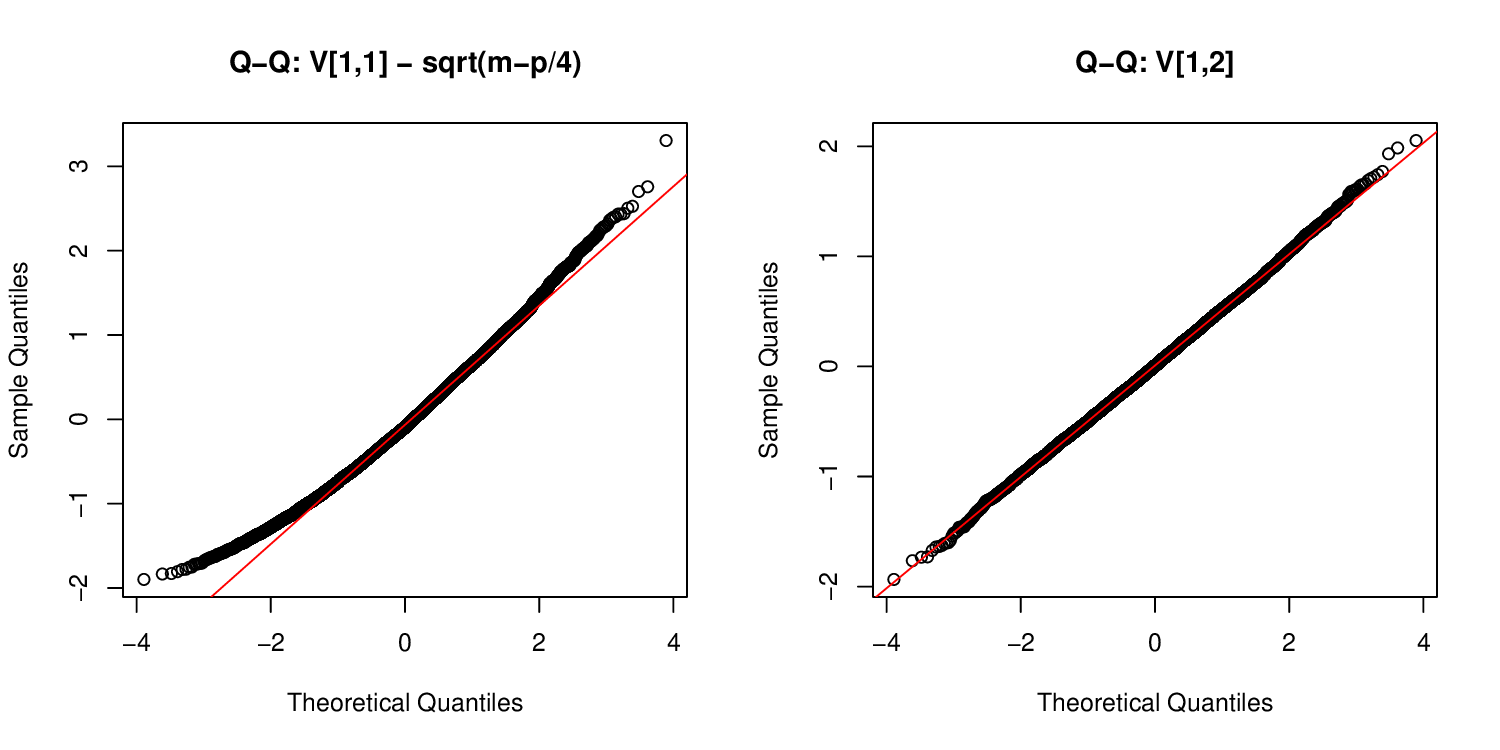}
        \caption{Q-Q plots comparing the empirical distributions of \( V_{11} - \sqrt{m-\frac{p}{4}} \) and \( V_{12} \) to their normal approximations }\label{52,QQ}
        
    \end{subfigure}
    \caption{Distribution and normal Q-Q comparison for $V_{11}-\sqrt{m-\frac{p}{4}}$ and $V_{12}$ with $(m,p)=(5,2)$}
  \label{vcom}
\end{figure}

\begin{table}[ht]
\centering
\caption{Empirical mean values of the entries of $V = \sqrt{W}$ for $p=2$, $m=5$}
\label{tab:mean}
\begin{tabular}{c|ccc}
       & $V_{11}$ & $V_{22}$ & $V_{12}$ \\
\hline
mean & 2.069     & 2.066 & -0.011 \\

\end{tabular}
\end{table}

\begin{table}[H]
\centering
\caption{Empirical covariance matrix of the entries of $V = \sqrt{W}$ for $p=2$, $m=5$}
\label{tab:covariance}
\begin{tabular}{c|ccc}
       & $V_{11}$ & $V_{22}$ & $V_{12}$ \\
\hline
$V_{11}$ & 0.480 & -0.030 & 0.006 \\
$V_{22}$ & -0.030 & 0.476 & 0.003 \\
$V_{12}$ & 0.006 & 0.003 & 0.246 \\
\end{tabular}
\end{table}
We next present numerical experiments for another parameter setting, $(m,p)$. 
\autoref{1251} and \autoref{1252} present the marginal distributions and normality diagnostics for two representative entries of the matrix 
$V=\sqrt{W}$
  when 
$(m,p)=(12,5)$. \autoref{1251} shows the empirical density of 
$V_{11}-\sqrt{m-\frac{p}{4}}$
  and 
$V_{12}$
 , overlaid with their corresponding theoretical normal density curves, $\mathcal{N}(0,\frac{1}{2})$ and $\mathcal{N}(0,\frac{1}{4})$, respectively. 
The distribution of 
$V_{11}-\sqrt{m-\frac{p}{4}}$
  is approximately normal with mean 0 and variance 
close to $0.5$
 , while 
$V_{12}$
  closely follows a normal distribution with variance approximately 
$0.25$. The Q-Q plots in \autoref{1252} further confirm the near-normality of these two random variables.

Tables \ref{tab:mean2} and \ref{tab:covariance2} report the empirical mean and covariance matrix of selected entries of 
$V$. The means of the diagonal entries
$V_{jj},\ \forall 1\le j\le p$
  are close to 
$\sqrt{m-\frac{p}{4}}\approx3.28$, while the off-diagonal entries 
$V_{1j}, \ \forall 1< j\le p$ have means close to 0, as expected. 
The covariance matrix of the first nine entries supports the approximation that these components behave approximately as independent normal variables, with variances around 0.5 for diagonal elements and 0.25 for off-diagonal elements.
These results indicate that the approximation continues to perform very well even as 
$p$ increases.

 \begin{figure}[H]
    \centering
    \begin{subfigure}[t]{0.48\linewidth}
        \centering
        \includegraphics[width=\linewidth]{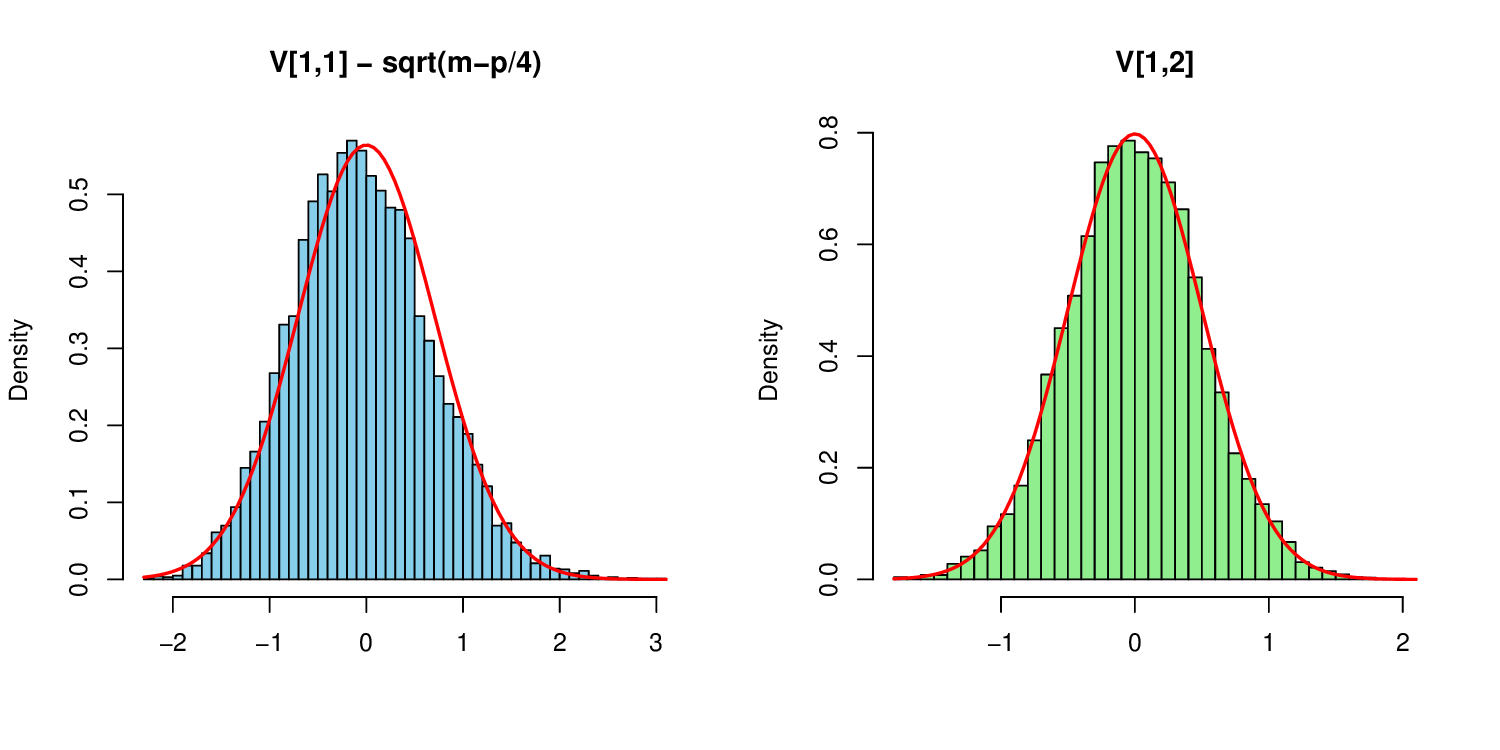}
        \caption{Comparison of empirical densities  of \( V_{11} - \sqrt{m-\frac{p}{4}} \) and \( V_{12} \), overlaid with \( \mathcal{N}(0, 1/2) \) and \( \mathcal{N}(0, 1/4) \), respectively}
        \label{1251}
    \end{subfigure}
    \hfill
    \begin{subfigure}[t]{0.48\linewidth}
        \centering
        \includegraphics[width=\linewidth]{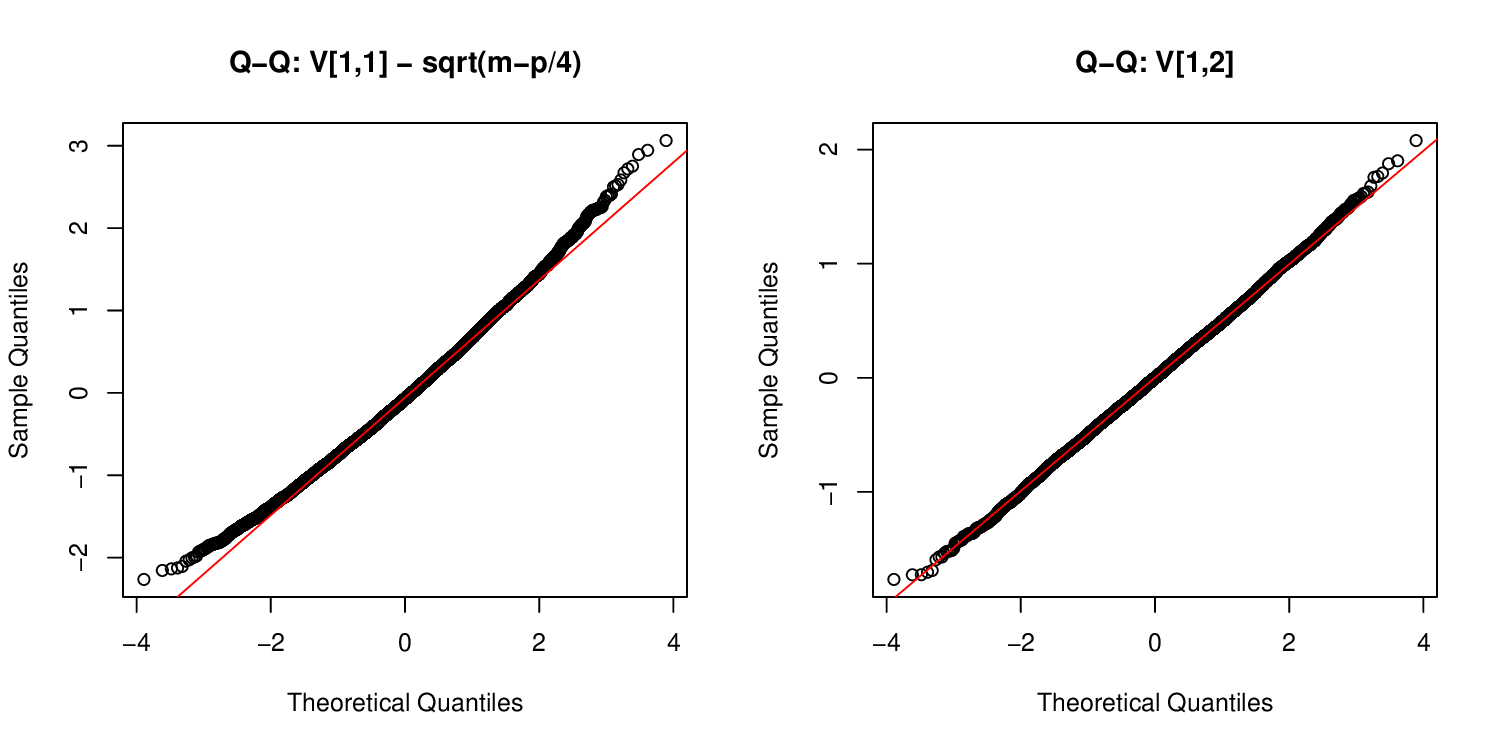}
        \caption{Q-Q plots comparing the empirical distributions of \( V_{11} - \sqrt{m-\frac{p}{4}} \) and \( V_{12} \) to their normal approximations}\label{1252}
        
    \end{subfigure}
    \caption{Distribution and normal Q-Q comparison for $V_{11}-\sqrt{m-\frac{p}{4}}$ and $V_{12}$ with $(m,p)=(12,5)$}
  \label{vcom}
\end{figure}

\begin{table}[H]
\centering
\caption{Empirical mean values of the first nine entries of $V = \sqrt{W}$ for $p=5$, $m=12$}
\label{tab:mean2}
\begin{tabular}{c|ccccccccc}
       & $V_{11}$ & $V_{22}$ &$V_{33}$ &$V_{44}$ &$V_{55}$& $V_{12}$ & $V_{13}$& $V_{14}$& $V_{15}$\\
\hline
mean& 3.230 & 3.241 & 3.230 & 3.246 & 3.232  &0.000 &-0.001 & 0.005& -0.003 \\

\end{tabular}
\end{table}
\begin{table}[H]
\centering
\caption{Empirical covariance matrix of the first nine entries of $V = \sqrt{W}$ for $p=5$, $m=12$}
\label{tab:covariance2}
\begin{tabular}{c|ccccccccc}
       & $V_{11}$ & $V_{22}$ &$V_{33}$ &$V_{44}$ &$V_{55}$& $V_{12}$ & $V_{13}$& $V_{14}$& $V_{15}$\\
\hline
$V_{11}$ & 0.501 &-0.005& -0.008& -0.010& -0.016&-0.006 &-0.009 & 0.007 &-0.002 \\
$V_{22}$ & -0.005 & 0.493& -0.002& -0.002& -0.014 & 0.004& -0.010 & 0.002 &-0.002 \\
$V_{33}$ & -0.008 &-0.002 & 0.497& -0.003& -0.012& -0.005  &0.001& -0.001 & 0.008 \\
$V_{44}$& -0.010& -0.002 &-0.003 & 0.497& -0.005 &-0.003 &-0.005 & 0.003 &-0.005\\
$V_{55}$&-0.016& -0.014 &-0.012 &-0.005 & 0.497 &-0.002  &0.001  &0.005& -0.002 \\
$V_{12}$&  -0.006 & 0.004& -0.005& -0.003 &-0.002 & 0.260& -0.001 &-0.006 &-0.001\\
$V_{13}$& -0.009 &-0.010 & 0.001 &-0.005&  0.001 &-0.001 & 0.259 & 0.002 & 0.002\\
$V_{14}$& 0.007 & 0.002 &-0.001 & 0.003 & 0.005& -0.006 & 0.002 & 0.254 & 0.004 \\
$V_{15}$& -0.002& -0.002  &0.008& -0.005 &-0.002& -0.001 & 0.002 & 0.004 & 0.255\\

\end{tabular}
\end{table}
We further examine the empirical covariances between selected entries for other combinations of $(m,p)$, as shown in Table \ref{tab:covariance3}.  Across all cases, the correlations between distinct elements remain close to zero, confirming that the asymptotic independence persists in higher-dimensional settings.

 \begin{table}[H]
\centering
\caption{Empirical covariance of certain pairs of entries for different $(m,p)$}
\label{tab:covariance3}
\begin{tabular}{c|cccccc}
    $(m,p)$   & cov($V_{11}$,$V_{11}$) & cov($V_{11}$,$V_{12}$) & cov($V_{12}$,$V_{13}$)& cov($V_{12}$,$V_{12}$) &cov($V_{11}$,$V_{22}$)\\
\hline
$(20,10)$ & 0.522  & 0.003 &-0.003& 0.253&-0.009\\
$(100,50)$ & 0.526 & -0.001 & 0.000&0.256&-0.007 \\
$(200,100)$&0.528 &-0.009   &0.004&0.260 &-0.002\\

\end{tabular}
\end{table}

\section{Conclusion}
In this paper, we have investigated both the explicit density function and the asymptotic distribution of the (symmetric) square root of a Wishart matrix.
We established that it converges to a multivariate normal distribution corresponding to one-half of a Gaussian Wigner matrix.
Although similar asymptotic results can be obtained using the delta method, our approach has the distinct advantage of yielding explicit convergence rates, which we show to be of order 
$O(p^{2.5}/\sqrt{m})$. The simulation studies corroborate our theoretical findings, demonstrating that the empirical distributions align closely with the limiting Gaussian law even when the degrees of freedom 
$m$ are only moderately larger than 
$p$.
This suggests that the asymptotic regime is reached rapidly, enhancing the practical relevance of the proposed approximation. 
The analytical framework developed here may also be extended to other matrix functionals in high-dimensional statistics.
Future research directions include deriving sharper upper bounds for the 1-Wasserstein distance and investigating convergence rates under alternative probabilistic metrics such as the total variation and Kolmogorov distances.

\section*{Acknowledgments }

The author  thanks Professor Yaning Yang for valuable comments and suggestions.

\end{document}